\documentclass[letterpaper, 10 pt, conference]{ieeeconf}

\IEEEoverridecommandlockouts
\usepackage{cite}

\usepackage{amsthm,amsmath,amssymb,amsfonts}
\usepackage{bm}  
\usepackage{algorithmic}
\usepackage[ruled, noline, linesnumbered]{algorithm2e}
\usepackage{graphicx}
\usepackage{xcolor}

\usepackage{array}
\usepackage{multirow}

\makeatletter
\let\MYcaption\@makecaption
\makeatother

\usepackage[font=footnotesize]{subcaption}

\makeatletter
\let\@makecaption\MYcaption
\makeatother

\usepackage{textcomp}
\def\BibTeX{{\rm B\kern-.05em{\sc i\kern-.025em b}\kern-.08em
    T\kern-.1667em\lower.7ex\hbox{E}\kern-.125emX}}

\usepackage{lipsum}


\theoremstyle{plain}

\newtheorem{lemma}{Lemma}
\newtheorem{corollary}{Corollary}

\theoremstyle{definition}

\def\({\left(}
\def\){\right)}
\def\[{\left[}
\def\]{\right]}

\def\abf{{\bf a}}
\def\bbf{{\bf b}}
\def\dbf{{\bf d}}
\def\Gbf{{\bf G}}  
\def\Jbf{{\bf J}}  
\def\Kbf{{\bf K}}  
\def\Mbf{{\bf M}}  
\def\Nbf{{\bf N}}  
\def\Pbf{{\bf P}}  
\def\Qbf{{\bf Q}}  
\def\Rbf{{\bf R}}  
\def\Sbf{{\bf S}}  
\def\Tbf{{\bf T}}  
\def\Ubf{{\bf U}}  \def\ubf{{\bf u}}
\def\Vbf{{\bf V}}  
\def\Wbf{{\bf W}}  \def\wbf{{\bf w}}
\def\xbf{{\bf x}}
\def\Ybf{{\bf Y}}  \def\ybf{{\bf y}}
\def\Zbf{{\bf Z}}  
\def\Deltabf{\bm{\Delta}}  \def\deltabf{\bm{\delta}}
\def\etabf{{\bm{\eta}}}
\def\xibf{{\bm{\xi}}}
\def\Phibf{\bm{\Phi}}
\def\Ccal{\mathcal{C}}  
\def\Hcal{\mathcal{H}}  
\def\Rcal{\mathcal{R}}  

\newif\ifshowWriterComment
\newcommand{\SFpair}{\{\Phibf_\xbf, \Phibf_\ubf\}}
\newcommand{\OFqple}{\{\Phibf_{\xbf\xbf}, \Phibf_{\ubf\xbf}, \Phibf_{\xbf\ybf}, \Phibf_{\ubf\ybf}\}}

\def\mat#1{\begin{bmatrix}#1\end{bmatrix}}
\def\t{[t]}
\def\tm{[t-1]}
\def\cor#1{Corollary~\ref{cor:#1}}
\def\fig#1{Fig.~\ref{fig:#1}}
\def\lem#1{Lemma~\ref{lem:#1}}
\def\sec#1{Section~\ref{sec:#1}}
\def\eqn#1{\eqref{eqn:#1}}

\def\st{{\rm s.t.}}
\def\OptConsSep{&&\quad}

\newcommand\OptCons[3]{
&\ #1 
\ifx\\#2\\ \else \OptConsSep #2 \fi%
\ifx\\#3\\ \nonumber \else \label{eqn:#3} \fi%
}

\newcommand{\OptMinN}[2]{
\begin{alignat*}{2}
\min\ &\ #1 \\
\st\ #2 
\end{alignat*}
}

\title{\LARGE \bf Realization, Internal Stability, and Controller Synthesis}

\author{Shih-Hao Tseng
\thanks{Shih-Hao Tseng is with the Division of Engineering and Applied Science, California Institute of Technology, Pasadena, CA 91125, USA.  Email: {\tt\small shtseng@caltech.edu}}
}

\showWriterCommenttrue

\begin{document}

\maketitle
\thispagestyle{empty}
\pagestyle{empty}

\bstctlcite{IEEE_BSTcontrol}

\begin{abstract}

We have witnessed the emergence of several controller parameterizations and the corresponding synthesis methods, including Youla, system level, input-output, and many other new proposals. Meanwhile, under the same synthesis method, there are multiple realizations to adopt. Different synthesis methods/realizations target different plants/scenarios. Except for some case-by-case studies, we don't currently have a unified framework to understand their relationships.

To address the issue, we show that existing controller synthesis methods and realization proposals are all special cases of a simple lemma, the realization-stability lemma. The lemma leads to easier equivalence proofs among existing methods and enables the formulation of a general controller synthesis problem, which provides a unified foundation for controller synthesis and realization derivation. 

\end{abstract}

\section{Introduction}\label{sec:introduction}

Synthesizing an internally stabilizing controller is a daunting task, especially for large-scale, complex, networked systems with multiple-input multiple-output plants. A well-celebrated pioneer work on controller synthesis is by Youla et al. \cite{youla1976modern1, youla1976modern2}, which shows that the set of all internally stabilizing controllers can be parameterized using a coprime factorization approach. One drawback of Youla parameterization is the difficulty of imposing structural constraints -- the constraints could only be imposed (while maintaining convexity
) 
in an intricate form admitting quadratic invariant property \cite{rotkowitz2005characterization,sabau2014youla,lessard2015convexity}.
To address this issue, system level parameterization (SLP) \cite{anderson2019system,wang2019system} proposes to work on the closed-loop system response and the corresponding system level synthesis (SLS) method can easily incorporate multiple structural constraints into a much simpler convex program \cite{wang2016localized,anderson2017structured}.
The success of SLP triggers the study of affine space parameterization of internally stabilizing controllers. \cite{furieri2019input} shows that the set of internally stabilizing controllers can also be parameterized in an input-output manner using the input-output parameterization (IOP).
Though a recent paper shows that Youla, SLP, and IOP are equivalent \cite{zheng2020equivalence}, there are still new affine space parameterizations found \cite{zheng2019system}.

Given the flourishing development of novel parameterizations and their corresponding synthesis methods, we have some natural questions to ask: Have we exhausted all possible parameterizations? Will we discover new synthesis methods? If so, why would they be the way they are? And, perhaps more importantly, how could we find/understand them systematically?

To add to this already puzzling situation, we have seen new results on realizations. Realizations, or block diagrams/implementations,\footnote{We adopt the terminology in \cite{tseng2020deployment,tsengsubsynthesis} that distinguishes ``realizations'' from ``implementations,'' where the former refers to the block diagrams (mathematical expressions) and the latter is reserved for the physical architecture consisting of computation, memory, and communication units.} 
describe how a system can be built from some interconnection of basic blocks/transfer functions.
It is well known, also shown by recent studies \cite{tseng2020deployment,li2020separating}, that the same controller can admit multiple different realizations, even under the same parameterization scheme. 
We would then wonder if we can only handle those realizations case by case, or if there is a unified framework to study them.

\subsection{Contributions and Organization}

The main contribution of this paper is the answers to all of the above seemingly unrelated questions through a simple \emph{realization-stability lemma} that relates closed-loop realizations with internal stability. The lemma enables us to formulate the general controller synthesis problem that can derive \emph{all} possible parameterizations, thus providing a systematic way to study controller synthesis problems. We show that existing methods on controller synthesis and realization are all special cases of the general formulation. In addition, the lemma reveals that the transformation of external disturbances can be seen as the derivation of an equivalent system. The concept of equivalent systems then enables easy proof of equivalence among synthesis methods.

The paper is organized as follows. In~\sec{R-S_lemma}, we derive the realization-stability lemma, introduce equivalent systems under transformations, and formulate the general controller synthesis problem. We then show in~\sec{controller_synthesis} and \sec{realizations} that existing methods are all special cases of the general framework and direct applications of the realization-stability lemma. In~\sec{controller_synthesis}, we revisit controller synthesis theories, including Youla \cite{youla1976modern2}, input-output \cite{furieri2019input}, system level \cite{wang2019system,anderson2019system}, and some mixed parameterizations \cite{zheng2019system}, and verify the equivalence results in~\cite{zheng2020equivalence}. In~\sec{realizations}, we derive the original SLS realization and alternative realizations from \cite{tseng2020deployment} and \cite{li2020separating} using the realization-stability lemma. Finally, we conclude the paper in~\sec{conclusion}.

\subsection{Notation} 

\def\Rp{\Rcal_{p}}
\def\Rsp{\Rcal_{sp}}
\def\RHinf{\Rcal\Hcal_{\infty}}

Let $\Rp$, $\Rsp$, and $\RHinf$ denote the set of proper, strictly proper, and stable proper transfer matrices, respectively, all defined according to the underlying setting, continuous or discrete.
Lower- and upper-case letters (such as $x$ and $A$) denote vectors and matrices respectively, while bold lower- and upper-case characters and symbols (such as $\ubf$ and $\Rbf$) are reserved for signals and transfer matrices. We denote by $I$ and $O$ the identity and all-zero matrices (with dimensions defined according to the context). 

\def\xx{\xbf\xbf}  \def\xu{\xbf\ubf}  \def\xy{\xbf\ybf}
\def\ux{\ubf\xbf}  \def\uu{\ubf\ubf}  \def\uy{\ubf\ybf}
\def\yx{\ybf\xbf}  \def\yu{\ybf\ubf}  \def\yy{\ybf\ybf}

\def\xd{\xbf\deltabf}
\def\ud{\ubf\deltabf}
\def\dx{\deltabf\xbf}  \def\du{\deltabf\ubf}  \def\dd{\deltabf\deltabf}

\section{Realization-Stability Lemma}\label{sec:R-S_lemma}
To begin with, we define the realization and internal stability matrices to derive the realization-stability lemma. We then discuss the transformation of external disturbances and introduce the concept of equivalent systems. Using the realization-stability lemma, we propose the formulation of a general controller synthesis problem. 

We remark that the results in this section are general: They apply to both discrete-time and continuous-time systems.

\subsection{Realization and Internal Stability}\label{sec:R-S_lemma-RS}
\begin{figure}
\centering
\includegraphics[scale=1]{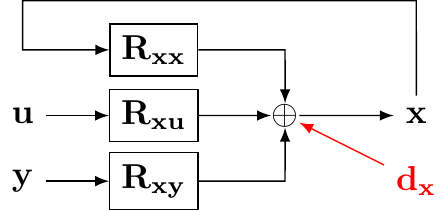}
\caption{The realization matrix $\Rbf$ describes each signal as a linear combination of the signals in the closed-loop system and the external disturbance $\dbf$. In the following figures of realizations, we omit drawing the additive disturbance $\dbf$ for simplicity.}
\label{fig:R-S-realization}
\end{figure}

We consider a closed-loop linear system with internal state $\etabf$ and external disturbance $\dbf$. The system operates according to the \emph{realization matrix} $\Rbf$:
\begin{align}
\etabf = \Rbf \etabf + \dbf.
\label{eqn:R}
\end{align}
$\etabf$ summarizes \emph{all} signals in the system. For instance, a state-feedback system might have $\etabf = \mat{\xbf\\ \xibf\\ \ubf}$ where $\xbf$ is the state, $\xibf$ is the internal state, and $\ubf$ is the control. For a given signal $\abf$, we denote by $e_{\abf}$ the column block that is identity at the rows corresponding to $\abf$ in $\etabf$. As a result,
$\etabf = \sum\limits_{\abf} e_{\abf} \abf$. 

$\Rbf$ describes each signal as a linear combination of the signals in the system. We denote by $\Rbf_{\abf\bbf}$ the transfer matrix block from signal $\bbf$ to $\abf$ as shown in~\fig{R-S-realization}, and hence given a signal $\abf$, we have $\abf = \sum\limits_{\bbf} \Rbf_{\abf\bbf} \bbf + \dbf_{\abf}$, where $\dbf_{\abf}$ is the external disturbance on $\abf$. Notice that all dimensions in the internal state $\etabf$ have their corresponding share in $\dbf$, thereby avoiding the partial selection issues discussed in \cite{zheng2020equivalence}.

On the other hand, if we deem the external disturbance $\dbf$ as the input and the internal state $\etabf$ as the output, we can treat the closed-loop system as an open-loop system. We denote by the \emph{internal stability matrix} (or \emph{stability matrix} for short) $\Sbf$ the transfer matrix of such an open-loop system:
\begin{align}
\etabf = \Sbf \dbf.
\label{eqn:S}
\end{align}
We define $\Sbf_{\abf\bbf}$ as the transfer matrix block from disturbance on $\bbf$ to the signal $\abf$, and the columns in $\Sbf$ corresponding to $\bbf$ is denoted by $\Sbf_{:,\bbf}$.

The realization matrix $\Rbf$ and the stability matrix $\Sbf$ are related by the following lemma.
\begin{lemma}[Realization-Stability]\label{lem:R-S}
Let $\Rbf$ be the realization matrix and $\Sbf$ be the internal stability matrix, we have
\begin{align*}
(I - \Rbf) \Sbf = \Sbf(I - \Rbf) = I.
\end{align*}
\end{lemma}

\begin{proof}
Substituting \eqn{S} into \eqn{R} yields
\begin{align*}
(I - \Rbf)\etabf = (I - \Rbf)\Sbf \dbf = \dbf.
\end{align*}
Since $\dbf$ is arbitrary, we have
\begin{align*}
(I - \Rbf)\Sbf = I.
\end{align*}
Given $I - \Rbf$ and $\Sbf$ are both square matrices, we have
\begin{align*}
\Sbf = (I-\Rbf)^{-1}  \quad \Rightarrow \quad
\Sbf (I - \Rbf) = I,
\end{align*}
which concludes the proof.
\end{proof}

We remark that \lem{R-S} does not guarantee the existence of either $\Rbf$ or $\Sbf$. Rather, it says if both $\Rbf$ and $\Sbf$ exist, they must obey the relation. When they both exist, a consequence of \lem{R-S} is that $\Rbf \to \Sbf$ is a bijection map. In other words, if two systems have the same realization $\Rbf$ (or $I-\Rbf$, equivalently), they have the same internal stability $\Sbf$. 

\subsection{Disturbance Transformation and Equivalent System}\label{sec:R-S_lemma-T}
In \eqn{R}, the external disturbance $\dbf$ affects each signal in the system independently. We can extend \eqn{R} and \eqn{S} to the cases where the dimensions in $\dbf$ are correlated. In particular, the external disturbance could be a \emph{transformation} $\Tbf$ on a different basis $\wbf$:
\begin{align*}
\dbf = \Tbf \wbf.
\end{align*}

When the transformation $\Tbf$ is invertible, we have
\begin{align*}
(I - \Rbf) \etabf = \Tbf \wbf
\quad\Rightarrow&\quad
\Tbf^{-1}(I - \Rbf) \etabf = \wbf = (I-\Rbf_{eq}) \etabf,\\
&\etabf = \Sbf \Tbf \wbf = \Sbf_{eq} \wbf
\end{align*}
where $\Rbf_{eq} = I - \Tbf^{-1}(I - \Rbf)$ and $\Sbf_{eq} = \Sbf \Tbf$.
In other words, the transformation of the disturbance $\dbf = \Tbf \wbf$ can be seen as the derivation of an equivalent closed-loop system with realization $\Rbf_{eq}$ and stability $\Sbf_{eq}$ based on internal state $\etabf$ and external disturbance $\wbf$.

The derivation of an equivalent system is helpful for stability analysis. Since \lem{R-S} suggests that there is a bijection map from $\Rbf$ to $\Sbf$. If there are two systems with realizations $\Rbf_1$ and $\Rbf_2$ and we can relate them through an (invertible) transformation $\Tbf$ by 
\begin{align*}
(I - \Rbf_2) = \Tbf^{-1}(I-\Rbf_1),
\end{align*}
their stability matrices will follow
\begin{align*}
\Sbf_2 = \Sbf_1 \Tbf.
\end{align*}

\subsection{Controller Synthesis and Column Dependency}\label{sec:R-S_lemma-controller-synthesis}
Notice that \lem{R-S} holds for arbitrary realization/internal stability matrices, e.g., non-causal $\Rbf$ and unstable $\Sbf$. When synthesizing a controller, we require the closed-loop system to be causal and internally stable. In other words, the transfer functions from one signal to any different signal should be proper, and the transfer functions from the external disturbance $\dbf$ to the internal state $\etabf$ should be stable proper, which are written as the following conditions:
\begin{align}
\Rbf_{\abf\bbf} \in \Rp, \forall \abf \neq \bbf, \quad\quad \Sbf \in \RHinf.
\label{eqn:R-S-condition}
\end{align}
Here, we implicitly require the existence of both $\Rbf$ and $\Sbf$. Accordingly, general controller synthesis problems (i.e., \emph{all} possible controller synthesis problems for a system described by some $\Rbf$) can be formulated as
\OptMinN{
g(\Rbf,\Sbf)
}{
\OptCons{(I-\Rbf)\Sbf = \Sbf(I-\Rbf) = I}{}{}\\
\OptCons{\Rbf_{\abf\bbf} \in \Rp}{\forall \abf \neq \bbf}{}\\
\OptCons{\Sbf \in \RHinf}{}{}\\
\OptCons{(\Rbf,\Sbf) \in \Ccal}{}{}
}
where $g$ is the objective function and $\Ccal$ represents the additional constraints on the realization and internal stability. 
In the following sections, we will show that the existing controller synthesis methods/realization studies that focus on internal stability are essentially special cases of the feasible set in this general formulation.

A key constraint in the general controller synthesis problem is to enforce $\Sbf \in \RHinf$. Although we need to enforce all elements in $\Sbf$ to be in $\RHinf$, we can leverage the linear dependency among the components brought by \lem{R-S} to derive some parts automatically without explicit enforcement. In particular, we have \lem{dependency}.

\begin{lemma}\label{lem:dependency}
Let $\abf$ be a signal and $\Rbf_{\abf\abf} = O$, then
\begin{align*}
\Sbf_{:,\abf} = e_{\abf} + \sum\limits_{\bbf \neq \abf} \Sbf_{:,\bbf}\Rbf_{\bbf\abf}.
\end{align*}
\end{lemma}

\begin{proof}
By \lem{R-S}, we have $\Sbf (I -\Rbf) = I$ and hence
\begin{align*}
\Sbf_{:,\abf}(I - \Rbf_{\abf\abf}) - \sum\limits_{\bbf \neq \abf} \Sbf_{:,\bbf}\Rbf_{\bbf\abf} = e_{\abf}.
\end{align*}
The lemma follows as $\Rbf_{\abf\abf} = O$.
\end{proof}

\lem{dependency} can greatly reduce the decision variables when synthesizing a controller. For instance, the synthesized control $\ubf$ is usually a function of other signals except for itself, which implies $\Rbf_{\uu} = O$. Therefore, \lem{dependency} gives
\begin{align}
\Sbf_{:,\ubf} = e_{\ubf} + \sum\limits_{\bbf \neq \ubf} \Sbf_{:,\bbf}\Rbf_{\bbf\ubf}.
\label{eqn:u-column}
\end{align}

\section{Corollaries: Controller Synthesis}\label{sec:controller_synthesis}
We use \lem{R-S} and condition \eqn{R-S-condition} to derive existing controller synthesis proposals, including Youla \cite{youla1976modern2}, input-output \cite{furieri2019input}, system level \cite{wang2019system,anderson2019system}, and mixed parameterizations \cite{zheng2019system}, with different $\Rbf$ and $\Sbf$ structures. We then demonstrate a simpler way to obtain the results in \cite{zheng2020equivalence} using transformations.

\begin{figure}
\centering
\includegraphics[scale=1]{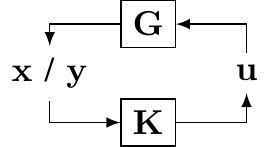}
\caption{The realization with plant $\Gbf$ and controller $\Kbf$. The internal signals include state $\xbf$ (or measurement $\ybf$) and control $\ubf$.}
\label{fig:realization-G-K}
\end{figure}

\subsection{Youla Parametrization}
Youla parameterization is based on the doubly coprime factorization of the plant $\Gbf$. If $\Gbf$ is stabilizable and detectable, we have
\begin{align*}
\mat{
\Mbf_l & -\Nbf_l\\
-\Vbf_l & \Ubf_l
}\mat{
\Ubf_r & \Nbf_r\\
\Vbf_r & \Mbf_r
} = I
\end{align*}
where both matrices are in $\RHinf$, $\Mbf_l$ and $\Mbf_r$ are both invertible in $\RHinf$, and $\Gbf = \Mbf_l^{-1} \Nbf_l = \Nbf_r \Mbf_r^{-1}$ \cite[Theorem 5.6]{zhou1998essentials}.

The following corollary is a modern rewrite of the original Youla parameterization in \cite[Lemma 3]{youla1976modern2} given by \cite[Theorem 11.6]{zhou1998essentials}:
\begin{corollary}
Let the plant $\Gbf$ be doubly coprime factorizable. Given $\Qbf \in \RHinf$, the set of all proper controllers achieving internal stability is parameterized by
\begin{align*}
\Kbf = (\Vbf_r - \Mbf_r \Qbf)(\Ubf_r - \Nbf_r \Qbf)^{-1}.
\end{align*}
\end{corollary}

\begin{proof}
Consider the realization in \fig{realization-G-K}, which has
\begin{align*}
\Rbf = \mat{O & \Gbf \\ \Kbf & O}, \quad \etabf = \mat{\xbf\\ \ubf}.
\end{align*}

To show that all $\Kbf$ can be parameterized by $\Qbf \in \RHinf$, we need to show that each $\Qbf$ is mapped to one valid $\Kbf$ and vice versa. For mapping $\Qbf$ to $\Kbf$, we consider the following transformation:
\begin{align*}
\Tbf^{-1} =&
\mat{
\Mbf_l^{-1} & O\\
O & (\Ubf_l -\Qbf\Nbf_l)^{-1}\\
}
\mat{
I & O\\
\Qbf & I
}\\
\Tbf =&
\mat{
I & O\\
-\Qbf & I
}
\mat{
\Mbf_l & O\\
O & \Ubf_l -\Qbf\Nbf_l \\
}
\end{align*}
As such,
\begin{align*}
I =&\ \Tbf^{-1} \mat{
\Mbf_l & -\Nbf_l\\
-\Vbf_l & \Ubf_l
}\mat{
\Ubf_r & \Nbf_r\\
\Vbf_r & \Mbf_r
}\Tbf\\
=&
\mat{I & -\Gbf \\ -\Kbf & I}
\mat{
(\Ubf_r - \Nbf_r \Qbf) \Mbf_l & \Sbf_{\xu}\\
(\Vbf_r - \Mbf_r \Qbf) \Mbf_l & \Sbf_{\uu}
} = (I -\Rbf)\Sbf
\end{align*}
where $\Sbf_{\xu}$ and $\Sbf_{\uu}$ are given by \eqn{u-column} and $\Gbf = \Mbf_l^{-1}\Nbf_l$:
\begin{align*}
\mat{\Sbf_{\xbf\ubf} \\ \Sbf_{\ubf\ubf}} =&
\mat{O\\ I} + 
\( \mat{\Ubf_r \\ \Vbf_r} -  \mat{\Nbf_r \\ \Mbf_r} \Qbf \) \Mbf_l \Gbf\\
=&
\mat{O\\ I} + 
\( \mat{\Ubf_r \\ \Vbf_r} -  \mat{\Nbf_r \\ \Mbf_r} \Qbf \) \Nbf_l.
\end{align*}

Since $\Tbf \in \RHinf$, we have $\Sbf \in \RHinf$. Therefore,
\begin{align*}
-\Kbf (\Ubf_r - \Nbf_r \Qbf) + (\Vbf_r - \Mbf_r \Qbf) = O,
\end{align*}
which leads to the desired $\Kbf$.

On the other hand, for mapping $\Kbf$ to $\Qbf$, internal stability of $\Kbf$ implies the corresponding $\Sbf_{\ux} \in \RHinf$. We compute $\Qbf$ by
\begin{align*}
\Qbf = \Mbf_r^{-1} (\Vbf_r - \Sbf_{\ux}\Mbf_l^{-1}),
\end{align*}
which is also in $\RHinf$ as $\Mbf_l$ and $\Mbf_r$ are both invertible in $\RHinf$ (i.e., $\Mbf_l^{-1}, \Mbf_r^{-1} \in \RHinf$), and all elements in $\Sbf$ can be expressed in $\Qbf$ using \lem{R-S}.
\end{proof}

\def\IOPqple{\{\Ybf,\Ubf,\Wbf,\Zbf\}}
\subsection{Input-Output Parametrization (IOP)}
Inspired by the system level approach in \cite{wang2019system}, \cite{furieri2019input} revisits the input-output system studied by Youla parameterization and proposes IOP as follows that does not depend on the doubly coprime factorization \cite[Theorem 1]{furieri2019input}.

\begin{corollary}
For the realization in \fig{realization-G-K} with $\Gbf \in \Rsp$, the set of all proper internally stabilizing controller is parameterized by $\IOPqple$ that lies in the affine subspace defined by the equations
\begin{align*}
\mat{I & -\Gbf}
\mat{\Ybf & \Wbf\\ \Ubf & \Zbf}
=&
\mat{I & O},\\
\mat{\Ybf & \Wbf\\ \Ubf & \Zbf}
\mat{-\Gbf\\ I}
=&
\mat{O\\I},\\
\Ybf, \Ubf, \Wbf, \Zbf \in&\ \RHinf,
\end{align*}
and the controller is given by $\Kbf = \Ubf \Ybf^{-1}$.
\end{corollary}

\begin{proof}
We can write down the realization matrix in~\fig{realization-G-K}:
\begin{align*}
\Rbf = \mat{O & \Gbf \\ \Kbf & O}, \quad \etabf = \mat{\ybf \\ \ubf}.
\end{align*}

Given $\Kbf$, the derivation of $\IOPqple$ is a direct consequence of \lem{R-S} and condition \eqn{R-S-condition}, which suggest
\begin{align}
I =&\ (I-\Rbf)\Sbf = \mat{I & -\Gbf \\ -\Kbf & I}
\mat{\Ybf & \Wbf\\ \Ubf & \Zbf} \label{eqn:IOP}\\
=&\ \Sbf(I-\Rbf) = 
\mat{\Ybf & \Wbf\\ \Ubf & \Zbf}\mat{I & -\Gbf \\ -\Kbf & I}, \nonumber\\
&\ \Ybf, \Ubf, \Wbf, \Zbf \in \RHinf. \nonumber
\end{align}

Conversely, given $\IOPqple$, \eqn{IOP} implies
\begin{align*}
\Ubf = \Kbf \Ybf \quad\Rightarrow\quad 
\Kbf = \Ubf \Ybf^{-1}.
\end{align*}
We need to verify that $\Ybf$ is invertible in $\Rp$ so that $\Kbf \in \Rp$. Given $\Gbf \in \Rsp$, we know that
\begin{align*}
\Ybf = I + \Gbf \Ubf = I + (zI-\Lambda)^{-1} \Jbf.
\end{align*}
for some matrix $\Lambda$ and $\Jbf \in \Rp$. As a result,
\begin{align*}
\Ybf^{-1} = I + \sum\limits_{k\geq 1}^{\infty} (zI-\Lambda)^{-k}\Jbf^k \in \Rp,
\end{align*}
which concludes the proof.
\end{proof}

\subsection{System Level Parametrization/Synthesis (SLP/SLS)}
System level synthesis (SLS) uses system level parameterization (SLP) to parameterize internally stabilizing controllers. There are two SLPs: for state-feedback and output-feedback systems, respectively. We discuss them below.

\begin{figure}
\centering
\includegraphics[scale=1]{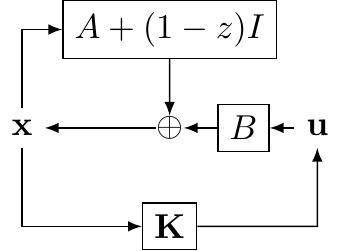}
\caption{The realization of a state-feedback system with controller $\Kbf$. The internal signals are state $\xbf$ and control $\ubf$.}
\label{fig:realization-state-feedback}
\end{figure}

\renewcommand{\paragraph}[1]{\vspace*{0.5\baselineskip}\noindent\textbf{#1}}

\paragraph{State-Feedback:} The following state-feedback parameterization is given in \cite[Theorem 1]{wang2019system} and \cite[Theorem 4.1]{anderson2019system}.

\begin{corollary}\label{cor:sf-SLS}
For the realization in~\fig{realization-state-feedback}, the set of all proper internally stabilizing state-feedback controller is parameterized by $\SFpair$ that lies in the affine space defined by
\begin{align*}
\mat{zI-A & -B}\mat{\Phibf_{\xbf}\\ \Phibf_{\ubf}} = I,\\
\Phibf_{\xbf}, \Phibf_{\ubf}\in z^{-1}\RHinf,
\end{align*}
and the controller is given by $\Kbf = \Phibf_{\ubf} \Phibf_{\xbf}^{-1}$.
\end{corollary}

\begin{proof}
The realization matrix in~\fig{realization-state-feedback} is
\begin{align}
\Rbf = \mat{
A + (1-z)I & B \\
\Kbf & O
}, \quad \etabf = \mat{\xbf\\ \ubf}.
\label{eqn:sf-SLS-R}
\end{align}

To derive $\Kbf$ from $\SFpair$, \lem{R-S} and condition \eqn{R-S-condition} lead to
\begin{align*}
\mat{
zI-A & -B \\
-\Kbf & I}&
\mat{\Phibf_{\xbf} & \Sbf_{\xu} \\ 
\Phibf_{\ubf} & \Sbf_{\uu}}
= I,\\
\Kbf \in \Rp,\quad &\Phibf_{\xbf}, \Phibf_{\ubf} \in \RHinf.
\end{align*}
Meanwhile, since
\begin{align*}
(zI-A) \Phibf_{\xbf} = I + B \Phibf_{\ubf} \in \RHinf,
\end{align*}
we have $\Phibf_{\xbf} \in z^{-1}\RHinf$. As a result, given $\Kbf \in \Rp$,
\begin{align*}
\Phibf_{\ubf} = \Kbf \Phibf_{\xbf} = z^{-1} \Kbf (z\Phibf_{\xbf}) \in z^{-1}\Rp,
\end{align*}
we know $\Phibf_{\ubf} \in z^{-1}\Rp \cap \RHinf = z^{-1}\RHinf$.

Conversely, given $\SFpair$, we can derive $\Kbf = \Phibf_{\ubf} \Phibf_{\xbf}^{-1}$ from \lem{R-S}. It remains to show that $\Sbf_{\xu}$ and $\Sbf_{\uu}$ exist whenever $\SFpair$ is given. According to \eqn{u-column}
\begin{align*}
\mat{\Sbf_{\xu}\\ \Sbf_{\uu}} = 
\mat{O\\I} + \mat{\Phibf_{\xbf}\\ \Phibf_{\ubf}}B \in \RHinf,
\end{align*}
which concludes the proof.
\end{proof}

\begin{figure}
\centering
\includegraphics[scale=1]{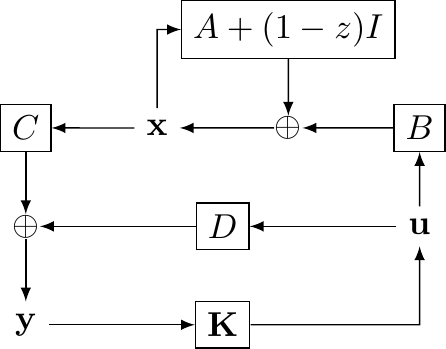}
\caption{The realization of an output-feedback system with controller $\Kbf$. The internal state $\etabf$ consists of state $\xbf$, control $\ubf$, and measurement $\ybf$ signals.}
\label{fig:realization-output-feedback}
\end{figure}

\paragraph{Output-Feedback:} The output-feedback SLP below is from \cite[Theorem 2]{wang2019system} and \cite[Theorem 5.1]{anderson2019system}.

\begin{corollary}\label{cor:of-SLS-zero-D}
For the realization in~\fig{realization-output-feedback} with $D = O$, the set of all proper internally stabilizing output-feedback controller is parameterized by $\OFqple$ that lies in the affine space defined by
\begin{align*}
\mat{
zI - A & -B
}
\mat{
\Phibf_{\xx} & \Phibf_{\xy}\\
\Phibf_{\ux} & \Phibf_{\uy}
}
=&
\mat{I & O},\\
\mat{
\Phibf_{\xx} & \Phibf_{\xy}\\
\Phibf_{\ux} & \Phibf_{\uy}
}
\mat{
zI - A \\ -C
}
=&
\mat{I \\ O},\\
\Phibf_{\xx}, \Phibf_{\xy}, \Phibf_{\ux} \in z^{-1}\RHinf,&\ \quad \Phibf_{\uy} \in \RHinf,
\end{align*}
and the controller is given by
\begin{align*}
\Kbf = \Phibf_{\uy} - \Phibf_{\ux} \Phibf_{\xx}^{-1} \Phibf_{\xy}.
\end{align*}

\end{corollary}

In fact, we can extend \cor{of-SLS-zero-D} to general $D$.
\begin{corollary}\label{cor:of-SLS}
Given $\OFqple$ that lies in the affine space in \cor{of-SLS-zero-D} and an arbitrary $D$, the proper internally stabilizing output-feedback controller $\Kbf$ is given by
\begin{align*}
\Kbf = \Kbf_0 \left( I + D \Kbf_0 \right)^{-1}
\end{align*}
where
$\Kbf_0 = \Phibf_{\uy} - \Phibf_{\ux} \Phibf_{\xx}^{-1} \Phibf_{\xy}$.
\end{corollary}

We prove the more general version -- \cor{of-SLS} -- below.

\begin{proof}
The realization matrix in~\fig{realization-output-feedback} is
\begin{align*}
\Rbf = \mat{
A + (1-z)I & B & O \\
O & O & \Kbf \\
C & D & O
}, \quad \etabf = \mat{\xbf\\ \ubf\\ \ybf}.
\end{align*}

Given $\Kbf$ we can directly derive $\OFqple$ from \lem{R-S}
\begin{align}
I =& \mat{
zI-A & -B & O \\
O & I & -\Kbf \\
-C & -D & I
}
\mat{
\Phibf_{\xx} & \Sbf_{\xu} & \Phibf_{\xy} \\
\Phibf_{\ux} & \Sbf_{\uu} & \Phibf_{\uy} \\
\Sbf_{\yx} & \Sbf_{\yu} & \Sbf_{\yy}
} \label{eqn:of-SLS}\\
=& \mat{
\Phibf_{\xx} & \Sbf_{\xu} & \Phibf_{\xy} \\
\Phibf_{\ux} & \Sbf_{\uu} & \Phibf_{\uy} \\
\Sbf_{\yx} & \Sbf_{\yu} & \Sbf_{\yy}
}
\mat{
zI-A & -B & O \\
O & I & -\Kbf \\
-C & -D & I
},\nonumber
\end{align}
where $\Sbf \in \RHinf$ by condition \eqn{R-S-condition}. As a result, we have
\begin{align*}
(zI-A)\Phibf_{\xx} =&\ I + B\Phibf_{\ux} \in \RHinf,\\
(zI-A)\Phibf_{\xy} =&\ B\Phibf_{\uy} \in \RHinf,\\
\Phibf_{\ux}(zI-A) =&\ \Phibf_{\uy}C \in \RHinf.
\end{align*}
Therefore, $\Phibf_{\xx},\Phibf_{\ux},\Phibf_{\xy} \in z^{-1}\RHinf$ and $\Phibf_{\uy} \in \RHinf$. 

Conversely, we can derive $\Kbf$ from $\OFqple$ as follows. First, we multiply the matrix
\begin{align*}
\Gamma = \mat{
I & B & O\\
O & I & O\\
O & D & I
} 
\end{align*}
at the left of both sides of \eqn{of-SLS}, which leads to
\begin{align*}
\mat{
zI-A & -B \Kbf \\
-C & I - D \Kbf
}
\mat{
\Phibf_{\xx} & \Phibf_{\xy} \\
\Sbf_{\yx} & \Sbf_{\yy}
} = I.
\end{align*}
Therefore, as $\Phibf_{\xx}$ and $\Sbf_{\yy}$ are both square, taking matrix inverse, we have
\begin{align*}
I - D \Kbf = \left( \Sbf_{\yy} -  \Sbf_{\yx} \Phibf_{\xx}^{-1} \Phibf_{\xy} \right)^{-1}.
\end{align*}
Since $\Phibf_{\ux} = \Kbf \Sbf_{\yx}$ and $\Phibf_{\uy} = \Kbf \Sbf_{\yy}$,
we know
\begin{align*}
\Kbf_0 = \Kbf \left( \Sbf_{\yy} -  \Sbf_{\yx} \Phibf_{\xx}^{-1} \Phibf_{\xy} \right)
\end{align*}
and we can rearrange the equation to obtain
\begin{align*}
\Kbf_0 - \Kbf D \Kbf_0 = \Kbf
\quad \Rightarrow \quad
\Kbf = \Kbf_0 \left( I + D \Kbf_0 \right)^{-1}.
\end{align*}

The last thing we need to verify is that $\Sbf$ exists and is in $\RHinf$. By \lem{R-S}, we know
\begin{align}
\Sbf_{\yx} =&\ C \Phibf_{\xx} + D \Phibf_{\ux} \in \RHinf, \nonumber \\
\Sbf_{\yy} =&\ C \Phibf_{\xy} + D \Phibf_{\uy} + I \in \RHinf. 
\label{eqn:of-SLS-S-1}
\end{align}
and we can compute the rest by \eqn{u-column}
\begin{align}
\mat{\Sbf_{\xu}\\ \Sbf_{\uu}\\ \Sbf_{\yu}} = 
\mat{O\\I\\O} + 
\mat{\Phibf_{\xx}\\ \Phibf_{\ux}\\ \Sbf_{\yx}}B +
\mat{\Phibf_{\xy}\\ \Phibf_{\uy}\\ \Sbf_{\yy}}D \in \RHinf,
\label{eqn:of-SLS-S-2}
\end{align}
which concludes the proof.
\end{proof}

\subsection{Mixed Parameterizations}
Letting $\Gbf = C(zI-A)^{-1}B + D$, \cite[Proposition 3, Proposition 4]{zheng2019system} provides the following corollaries that have conditions in both SLP and IOP flavors.

\begin{corollary}\label{cor:mixed-1}
For the realization in~\fig{realization-output-feedback}, the set of all proper internally stabilizing output-feedback controller is parameterized by $\{\Phibf_{\yx},\Phibf_{\ux},\Phibf_{\yy},\Phibf_{\uy}\}$ that lies in the affine space defined by
\begin{align*}
\mat{
I & -\Gbf
}
\mat{
\Phibf_{\yx} & \Phibf_{\yy}\\
\Phibf_{\ux} & \Phibf_{\uy}
}
=&
\mat{C(zI-A)^{-1} & I},\\
\mat{
\Phibf_{\yx} & \Phibf_{\yy}\\
\Phibf_{\ux} & \Phibf_{\uy}
}
\mat{
zI - A \\ -C
}
=&
O,\\
\Phibf_{\yx},\Phibf_{\ux},\Phibf_{\yy}&,\Phibf_{\uy} \in \RHinf,
\end{align*}
and the controller is given by
\begin{align*}
\Kbf = \Phibf_{\uy}\Phibf_{\yy}^{-1}.
\end{align*}
\end{corollary}

\begin{corollary}\label{cor:mixed-2}
For the realization in~\fig{realization-output-feedback}, the set of all proper internally stabilizing output-feedback controller is parameterized by $\{\Phibf_{\xy},\Phibf_{\uy},\Phibf_{\xu},\Phibf_{\uu}\}$ that lies in the affine space defined by
\begin{align*}
\mat{
zI-A & -B
}
\mat{
\Phibf_{\xy} & \Phibf_{\xu}\\
\Phibf_{\uy} & \Phibf_{\uu}
}
=&
O,\\
\mat{
\Phibf_{\yx} & \Phibf_{\yy}\\
\Phibf_{\ux} & \Phibf_{\uy}
}
\mat{
-\Gbf \\ I
}
=&
\mat{
(zI - A)^{-1}B \\ I
},\\
\Phibf_{\xy},\Phibf_{\uy},\Phibf_{\xu},\Phibf_{\uu} \in&\ \RHinf,
\end{align*}
and the controller is given by
\begin{align*}
\Kbf = \Phibf_{\uu}^{-1}\Phibf_{\uy}.
\end{align*}
\end{corollary}

We give a brief proof below for the two corollaries above.

\begin{proof}

\lem{R-S} gives
\begin{align*}
I =&\ (I-\Rbf)\Sbf\\
=&
\mat{
zI-A & -B & O \\
O & I & -\Kbf \\
-C & -D & I
}
\mat{
\Sbf_{\xx} & \Phibf_{\xu} & \Phibf_{\xy} \\
\Phibf_{\ux} & \Phibf_{\uu} & \Phibf_{\uy} \\
\Phibf_{\yx} & \Sbf_{\yu} & \Phibf_{\yy}
}.
\end{align*}
We consider two matrices
\begin{align*}
\Gamma_1 =& \mat{
I & O & O\\
O & I & O\\
C(zI-A)^{-1} & O & I
},\\
\Gamma_2 =& \mat{
I & (zI-A)^{-1}B & O\\
O & I & O\\
O & O & I
}.
\end{align*}

Analogous to the proof of \cor{of-SLS}, \cor{mixed-1} can be derived from the following conditions and condition \eqn{R-S-condition}.
\begin{align*}
\Gamma_1(I-\Rbf)\Sbf = \Gamma_1, \quad
\Sbf(I-\Rbf) = I.
\end{align*}
Similarly, we derive \cor{mixed-2} from condition \eqn{R-S-condition} and
\begin{align*}
(I-\Rbf)\Sbf = I, \quad
\Sbf(I-\Rbf)\Gamma_2 = \Gamma_2.
\end{align*}
\end{proof}

\subsection{Equivalence among Synthesis Methods}

The parameterizations above are shown equivalent in \cite{zheng2020equivalence} through careful calculations. Here we demonstrate how \lem{R-S} and transformations lead to more straightforward derivations of equivalent components. 

\lem{R-S} implies that $\Rbf \to \Sbf$ is a one-to-one mapping. Therefore, to show the equivalence among different synthesis methods, we can simply find a transformation $\Tbf$ such that the equivalent system has the same realization as the other system. As such, \lem{R-S} suggests that the stability matrices are the same, and we just need to compare the elements correspondingly.

When comparing a state-feedback system with an output-feedback system in the following analyses, we assume that the state $\xbf$ is taken as the measurement $\ybf$.

\paragraph{Youla parameterization and IOP:} Youla parameterization and IOP share the same realization in~\fig{realization-G-K} (except for changing $\xbf$ to $\ybf$). Therefore,
\begin{align*}
\mat{
\Ybf & \Wbf\\
\Ubf & \Zbf
}
=&
\mat{
\Ubf_r & \Nbf_r\\
\Vbf_r & \Mbf_r
}
\mat{
I & O\\
-\Qbf & I
}
\mat{
\Mbf_l & O\\
O & \Ubf_l -\Qbf\Nbf_l \\
}
\\
=&
\mat{
(\Ubf_r - \Nbf_r \Qbf) \Mbf_l & (\Ubf_r - \Nbf_r \Qbf) \Nbf_l \\
(\Vbf_r - \Mbf_r \Qbf) \Mbf_l & I + (\Vbf_r - \Mbf_r \Qbf) \Nbf_l
}.
\end{align*}

\paragraph{IOP and SLP:} We then show the equivalence between IOP and SLP. For state-feedback SLP with realization in~\fig{realization-state-feedback}, we perform the transformation
\begin{align*}
\Tbf^{-1} = \mat{(zI-A)^{-1} & O\\ O & I}, \quad 
\Tbf = \mat{zI-A & O\\ O & I},
\end{align*}
which leads to
\begin{align*}
\Tbf^{-1} \mat{
zI-A & -B \\
-\Kbf & I} = 
\mat{I & -\Gbf \\ -\Kbf & I}.
\end{align*}
Accordingly, the stability matrix becomes
\begin{align*}
\mat{\Ybf & \Wbf \\ \Ubf & \Zbf} =& \mat{\Phibf_{\xbf} & \Phibf_{\xbf}B \\ \Phibf_{\ubf} & I + \Phibf_{\ubf}B }\Tbf \\
=&
\mat{\Phibf_{\xbf}(zI-A) & \Phibf_{\xbf}B(zI-A) \\ \Phibf_{\ubf} & I + \Phibf_{\ubf}B }.
\end{align*}

For output-feedback SLP, we consider the transformation
\begin{align*}
\Tbf^{-1} =& \mat{
I & O & O\\
O & I & O\\
C(zI-A)^{-1} & O & I
},\\
\Tbf =& \mat{
I & O & O\\
O & I & O\\
-C(zI-A)^{-1} & O & I
},
\end{align*} 
which leads to
\begin{align*}
\Tbf^{-1} (I - \Rbf) =&\ \Tbf^{-1}
\mat{
zI-A & -B & O \\
O & I & -\Kbf \\
-C & -D & I
}\\
=&
\mat{
zI-A & -B & O\\
O & I & -\Kbf \\
O & -\Gbf & I
}, \quad \etabf = \mat{\xbf\\ \ubf\\ \ybf}.
\end{align*}
And the transformed stability matrix is
\begin{align*}
\Sbf\Tbf =&
\mat{
\Phibf_{\xx} & \Sbf_{\xu} & \Phibf_{\xy} \\
\Phibf_{\ux} & \Sbf_{\uu} & \Phibf_{\uy} \\
\Sbf_{\yx} & \Sbf_{\yu} & \Sbf_{\yy}
}\Tbf\\
=&
\mat{
\Phibf_{\xx} - \Phibf_{\xy}C(zI-A)^{-1} & \Sbf_{\xu} & \Phibf_{\xy} \\
\Phibf_{\ux} - \Phibf_{\uy}C(zI-A)^{-1} & \Sbf_{\uu} & \Phibf_{\uy} \\
\Sbf_{\yx} - \Sbf_{\yy}C(zI-A)^{-1} & \Sbf_{\yu} & \Sbf_{\yy}
}.
\end{align*}

Comparing the corresponding elements and we have
\begin{align*}
\mat{\Ybf & \Wbf \\ \Ubf & \Zbf}
=&
\mat{\Sbf_{\yy} & \Sbf_{\yu}\\
\Phibf_{\uy} & \Sbf_{\uu}
}\\
=&
\mat{
C\Phibf_{\xy} + D\Phibf_{\uy} + I & \Sbf_{\yu}\\
\Phibf_{\uy} & \Phibf_{\ux}B + \Phibf_{\uy}D + I
}
\end{align*}
where
\begin{align*}
\Sbf_{\yu} = (C\Phibf_{\xx} + D\Phibf_{\ux})B + (C\Phibf_{\xy} + D\Phibf_{\uy} + I)D.
\end{align*}
Our result extends the $D = O$ case in \cite{zheng2020equivalence} to general $D$.

\paragraph{SLP and mixed parameterizations:} SLP and mixed parameterizations share the same realization \fig{realization-output-feedback}. Therefore, they also share the same stability matrix according to \lem{R-S}, i.e., $\Phibf_{\xu},\Phibf_{\uu},\Phibf_{\yx},$ and $\Phibf_{\yy}$ can be found in \eqn{of-SLS-S-1} and \eqn{of-SLS-S-2}.

\section{Corollaries: Realizations}\label{sec:realizations}
The same parameterization could admit multiple different realizations\footnote{We remark that once $\Sbf$ is fixed, $\Rbf$ is uniquely defined by \lem{R-S} (if existing). However, one parameterization may not include the whole $\Sbf$, and hence there are still some degrees of freedom for different realizations $\Rbf$.}. In this section, we consider the original state-feedback SLS realization and two alternative realization proposals for SLS. We show that the realizations can be derived from \lem{R-S} through transformations.

\subsection{Original State-Feedback SLS Realization}
SLP parameterizes all internally stabilizing controller $\Kbf$ for the state-feedback system in~\fig{realization-state-feedback}. Using the resulting $\SFpair$, SLS proposes to implement the controller as in~\fig{realization-sf-SLS}.

\begin{figure}
\centering
\includegraphics[scale=1]{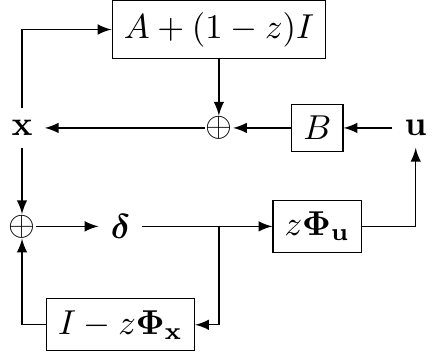}
\caption{The realization proposed in the original state-feedback SLS using the SLP $\SFpair$. By introducing an additional signal $\deltabf$, this realization avoids taking the inverse of $\Phibf_{\xbf}$.}
\label{fig:realization-sf-SLS}
\end{figure}

In other words, given $\Rbf$ as in \eqn{sf-SLS-R} and $\Sbf$ satisfying \lem{R-S}, we can realize the closed-loop system by
\begin{align*}
\Rbf_r = \mat{
A + (1-z)I & B & O\\
O & O & z\Phibf_{\ubf}\\
I & O & I - z\Phibf_{\xbf}
},
\quad
\etabf=\mat{
\xbf\\ \ubf\\ \deltabf
}.
\end{align*}

To show that, we augment \eqn{sf-SLS-R} with a dummy node
\begin{align}
I =&\ (I-\Rbf_{aug})\Sbf_{aug} \nonumber \\
=& 
\mat{
zI-A & -B & O\\
-\Kbf & I & O\\
O & O & I
}
\mat{
\Phibf_{\xbf} & \Sbf_{\xu} & O \\ 
\Phibf_{\ubf} & \Sbf_{\uu} & O \\
O & O & I 
}
\label{eqn:augmented}
\end{align}
and perform the following transformation on the augmented system to achieve the desired realization
\begin{align*}
\Tbf^{-1} =& 
\mat{
I & O & O\\
\Phibf_{\ubf} & \Sbf_{\uu} & -z\Phibf_{\ubf}\\
-\Phibf_{\xbf} & -\Sbf_{\xu} & z\Phibf_{\xbf}
},\\
\Tbf =& 
\mat{
I & O & O\\
O & I & \Kbf \\
z^{-1} & z^{-1}B & I - z^{-1}A
}.
\end{align*}
The realization is internally stable as
\begin{align*}
\Sbf_r =&\ \Sbf_{aug}\Tbf
=
\mat{
\Phibf_{\xbf} & \Sbf_{\xu} & \Sbf_{\xu}\Kbf \\ 
\Phibf_{\ubf} & \Sbf_{\uu} & \Sbf_{\uu}\Kbf \\
z^{-1} & z^{-1}B & I - z^{-1}A
}\\
=&
\mat{
\Phibf_{\xbf} & \Sbf_{\xu} & \Phibf_{\xbf}(zI-A) - I \\ 
\Phibf_{\ubf} & \Sbf_{\uu} & \Phibf_{\ubf}(zI-A) \\
z^{-1} & z^{-1}B & I - z^{-1}A
} \in \RHinf.
\end{align*}

\subsection{Simpler Realization for Deployment}
The original SLS realization in~\fig{realization-sf-SLS} needs to perform two convolutions $I - z\Phibf_{\xbf}$ and $z\Phibf_{\ubf}$, which are expensive to implement in practice. Therefore, \cite{tseng2020deployment} proposes a new realization in~\fig{realization-Tseng2020} that replaces one convolution by two matrix multiplications through the following corollary \cite[Theorem 1]{tseng2020deployment}.
\begin{figure}
\centering
\includegraphics[scale=1]{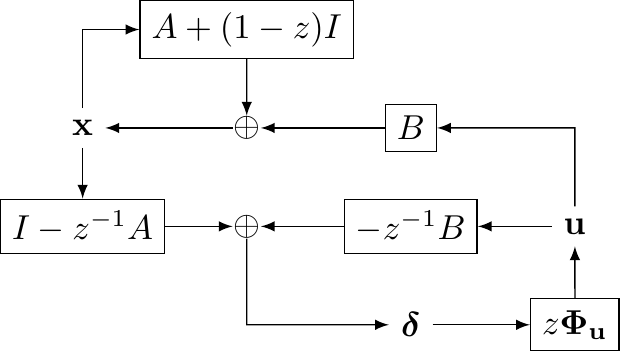}
\caption{The realization proposed in~\cite{tseng2020deployment} that realizes the SLS state-feedback controller using only one convolution $z\Phibf_{\ubf}$.}
\label{fig:realization-Tseng2020}
\end{figure}

\begin{corollary}
Let $A$ be Schur stable, the dynamic state-feedback controller $\ubf = \Kbf \xbf$ realized via
\begin{align*}
\delta\t =&\ x\t - Ax\tm - Bu\tm,\\
u\t =&\ \sum\limits_{\tau \geq 1} \Phi_{u}[\tau] \delta[t+1-\tau]
\end{align*}
is internally stabilizing.
\end{corollary}

\begin{proof}
We first write the controller realization in frequency domain:
\begin{align*}
\deltabf =&\ (I - z^{-1}A)\xbf - B\ubf,\\
\ubf =&\ z\Phibf_{\ubf}\deltabf.
\end{align*}
Together with the system, the realization is shown in~\fig{realization-Tseng2020}.

Essentially, the corollary says that given $\Rbf$ as in \eqn{sf-SLS-R} and $\Sbf$ satisfying \lem{R-S}, we can realize the closed-loop system by
\begin{align*}
\Rbf_{r} = \mat{
A+(1-z)I & B & O\\
O & O & z\Phibf_{\ubf}\\
z^{-1}(zI-A) & -z^{-1}B & O
}, \quad \etabf_{r} = \mat{\xbf\\ \ubf\\ \deltabf}.
\end{align*}

Again, we consider the augmented system in \eqn{augmented} and transform it to achieve the desired realization by
\begin{align*}
\Tbf^{-1} =& 
\mat{
I & O & O\\
\Phibf_{\ubf} & \Sbf_{\uu} & -z\Phibf_{\ubf}\\
-z^{-1} & O & I
},\\
\Tbf =& 
\mat{
I & O & O\\
O & \Sbf_{\uu}^{-1} & z\Sbf_{\uu}^{-1}\Phibf_{\ubf} \\
z^{-1} & O & I
}.
\end{align*}
As such, $(I-\Rbf_{r}) = \Tbf^{-1}(I-\Rbf_{aug})$ and the resulting stability matrix is
\begin{align*}
\Sbf_{r} = \Sbf_{aug}\Tbf = \mat{
\Phibf_{\xbf} & \Sbf_{\xu}\Sbf_{\uu}^{-1} & z\Sbf_{\xu}\Sbf_{\uu}^{-1}\Phibf_{\ubf} \\ 
\Phibf_{\ubf} & I & z\Phibf_{\ubf} \\
z^{-1} & O & I 
}.
\end{align*}
Since $\Sbf_{\uu} = I + \Phibf_{\ubf}B$ is invertible, $(zI-A)\Sbf_{\xu} = B\Sbf_{\uu}$, and $A$ is Schur stable, we have
\begin{align*}
\Sbf_{\xu}\Sbf_{\uu}^{-1} = (zI-A)^{-1}B \in \RHinf,
\end{align*}
and hence the stability matrix is in $\RHinf$.
\end{proof}

In \cite{tseng2020deployment}, the authors substitute $\ubf$ into $\deltabf$ before analyzing the internal stability, which is simply another (linear) transformation of $\dbf$ and the resulting stability matrix is still internally stable. 

\subsection{Closed-Loop Design Separation}
Instead of directly adopting the realization in~\fig{realization-sf-SLS}, \cite{li2020separating} found that it is possible to use much simpler transfer matrices to realize the same controller. The following corollary is from \cite[Theorem 2]{li2020separating}\footnote{To avoid the confusion with the realization matrix $\Rbf$, we write $\Pbf_c$ here instead.}.
\begin{figure}
\centering
\includegraphics[scale=1]{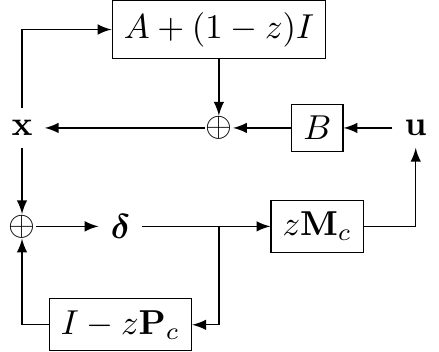}
\caption{The realization proposed in~\fig{realization-Li2020} that uses simpler transfer matrices $z\Mbf_c$ and $I-z\Pbf_c$ to implement the SLS state-feedback controller $\Phibf_{\ubf}\Phibf_{
\xbf}^{-1}$.}
\label{fig:realization-Li2020}
\end{figure}

\begin{corollary}\label{cor:Li2020}
For the causal realization in~\fig{realization-Li2020} and a given $\SFpair$ that satisfies \cor{sf-SLS}, $\{\Pbf_c,\Mbf_c\}$ realizes $\SFpair$ if and only if they satisfy
\begin{align}
\mat{
\Phibf_{\xbf} \\
\Phibf_{\ubf} 
}
\mat{
zI-A & -B
}
\mat{
\Pbf_c\\
\Mbf_c
}
=
\mat{
\Pbf_c\\
\Mbf_c
}.
\label{eqn:Li2020}
\end{align}
\end{corollary}

\begin{proof}
The corollary says that for the realization
\begin{align*}
\Rbf = \mat{
A + (1-z)I & B & O\\
O & O & z\Mbf_c\\
I & O & I - z\Pbf_c
},
\quad
\etabf=\mat{
\xbf\\ \ubf\\ \deltabf
}
\end{align*}
and the given $(\Phibf_{\xbf}, \Phibf_{\ubf})$, there exists a solution
\begin{align*}
\Sbf = \mat{
\Phibf_{\xbf} & \Sbf_{\xu} & \Sbf_{\xd}\\
\Phibf_{\ubf} & \Sbf_{\uu} & \Sbf_{\ud}\\
\Sbf_{\dx} & \Sbf_{\du} & \Sbf_{\dd}
}
\end{align*}
satisfying \lem{R-S} if and only if \eqn{Li2020} holds.

If such an $\Sbf$ exists, 
we have
\begin{align*}
(I-\Rbf)
\mat{
z\Pbf_c\\
z\Mbf_c\\
I
}
=
\mat{
zI-A & -B & O\\
O & O & O\\
O & O & O
}
\mat{
z\Pbf_c\\
z\Mbf_c\\
I
}.
\end{align*}
Therefore, by \lem{R-S}, $\Sbf (I-\Rbf) = I$ and 
\begin{align}
&\ \Sbf(I-\Rbf)\mat{
z\Pbf_c\\
z\Mbf_c\\
I
} = \mat{
z\Pbf_c\\
z\Mbf_c\\
I
} \label{eqn:Li2020-Sdx}\\
\Rightarrow&
\mat{
\Phibf_{\xbf} \\
\Phibf_{\ubf} 
}
\mat{
zI-A & -B
}
\mat{
z\Pbf_c\\
z\Mbf_c
}
=
\mat{
z\Pbf_c\\
z\Mbf_c
}.\nonumber
\end{align}
\eqn{Li2020} then follows from dividing both sides by $z$.

On the other hand, when \eqn{Li2020} holds and $\SFpair$ satisfies \cor{sf-SLS}, the stability matrix can be derived from $\Sbf(I-\Rbf) = I$ as
\begin{align*}
\Sbf =
 \mat{
\Phibf_{\xbf} & \Phibf_{\xbf}B & \Phibf_{\xbf}(zI-A) - I \\
\Phibf_{\ubf} & I + \Phibf_{\ubf}B & \Phibf_{\ubf}(zI-A) \\
\Sbf_{\dx} & \Sbf_{\dx}B & \Sbf_{\dx}(zI-A)
}
\end{align*}
where, by \eqn{Li2020-Sdx},
\begin{align*}
\Sbf_{\dx} = z^{-1} \Deltabf_c^{-1} = z^{-1}\left( \mat{
zI-A & -B
}
\mat{
\Pbf_c\\
\Mbf_c
} \right)^{-1}.
\end{align*}

$\Sbf$ exists if $\Sbf_{\dx}$ exists. In other words, we have to show that $\Deltabf_c$ is invertible. Since the system is causal, $I-z\Pbf_c$ and $z\Mbf_c$ are both in $\Rp$. Therefore,
\begin{align*}
\Deltabf_c = z\Pbf_c - (A\Pbf_c + B\Mbf_c) = I - z^{-1}\Jbf
\end{align*}
where $\Jbf \in \Rp$, and hence
\begin{align*}
\Deltabf_c^{-1} = I + \sum\limits_{k \geq 1}^{\infty}  z^{-k}\Jbf^k \in \Rp,
\end{align*}
which concludes the proof.
\end{proof}

We remark that \cor{Li2020} does not guarantee that $\Sbf \in \RHinf$, and hence the authors in \cite{li2020separating} propose to perform a posteriori stability check. According to the proof of \cor{Li2020}, we can easily guarantee $\Sbf \in \RHinf$ by requiring $\Sbf_{\dx} \in z^{-1}\RHinf$ (to ensure $\Sbf_{\dd} = \Sbf_{\dx}(zI-A) \in \RHinf$). This is one benefit resulting from the analysis using \lem{R-S} and condition \eqn{R-S-condition}.

\section{Conclusion}\label{sec:conclusion}
We derived the realization-stability lemma, introduced the concept of equivalent systems through transformation, and formulated the general controller synthesis problem. Several existing controller parameterization methods, including Youla parameterization, IOP, SLP, and some new mixed parameterizations, are all special cases of the general framework with different realizations. Existing realization results can also be derived from the lemma. Through these case studies, we demonstrate a unified procedure to perform controller synthesis and realization derivation.

\bibliographystyle{IEEEtran}
\bibliography{Test}
\end{document}